\documentclass[12pt]{amsart}
\usepackage{amsmath,amssymb}
\newtheorem{theorem}{Theorem}
\newtheorem{proposition}[theorem]{Proposition}
\newtheorem{lemma}[theorem]{Lemma}
\newtheorem{corollary}[theorem]{Corollary}
\def\R{\Bbb R}
\def\RR{\Bbb R^2_+}
\def\RRR{\Bbb R^n_+}
\def\om{\omega}
\def\tom{\tilde\omega}
\def\g{\gamma}
\def\ds{\displaystyle}

\title{Boundary behavior of the quasi-hyperbolic metric}

\author{Nikolai Nikolov and Pascal J. Thomas}

\address{N. Nikolov: Institute of Mathematics and Informatics\\Bulgarian Academy
of Sciences\\Acad. G. Bonchev 8, 1113 Sofia, Bulgaria\newline
\indent Faculty of Information Sciences\\
State University of Library Studies and Information Technologies\\
Shipchenski prohod 69A, 1574 Sofia, Bulgaria}\email{nik@math.bas.bg}

\address{Pascal J. Thomas\\
Universit\'e de Toulouse\\ UPS, INSA, UT1, UTM \\
Institut de Math\'e\-matiques de Toulouse\\
F-31062 Toulouse, France} \email{pascal.thomas@math.univ-toulouse.fr}



\thanks{This paper was started while the first-named author was visiting
the Paul Saba\-tier University, Toulouse in November 2015.}

\begin{document}

\begin{abstract}{The precise behavior of the quasi-hyperbolic metric
near a $\mathcal C^{1,1}$-smooth part of the boundary of a domain in $\R^n$ is obtained.}
\end{abstract}

\maketitle

\section{Introduction and results}

Let $D$ be a proper subdomain of  $\R^n.$ Define the {\it quasi-hyperbolic metric} of $D$ by
$$h_D(a,b)=\inf_{\g}\int_\g\frac{||du||}{d_D(u)},\quad a,b\in D$$
where $||\cdot||$ is the Euclidean norm, $d_D=\mbox{dist}(\cdot,\partial D)$ and the
infimum is taken over all rectifiable curves $\g$ in $D$ joining $a$ to $b.$
By \cite[Lemma 1]{GO}, the infimum is attained, and
any extremal curve is called {\it quasi-hyperbolic geodesic} (for short, {\it geodesic}).
It turns out that the geodesics are $\mathcal C^{1,1}$-smooth (see \cite[Corollary 4.8]{Mar}).
The quasi-hyperbolic metric arises in the theory of quasi-conformal maps.

This paper is devoted to the boundary behavior of $h_D.$ First, we point out
the following general lower bound.

\begin{proposition}\label{GHM}\cite[Lemma 2.6]{GHM} If $D$ is a proper subdomain of $\R^n,$ then
$$h_D(a,b)\ge 2\log\frac{d_D(a)+d_D(b)+||a-b||}{2\sqrt{d_D(a)d_D(b)}},\quad a,b\in D.$$
\end{proposition}

Observe that equality occurs if $n=1$ (then $D$ is an open interval or ray).

From now, we assume that $n\ge 2.$ Throughout the paper, we will say that
$\zeta$ is a \emph{$\mathcal C^\alpha$ smooth boundary point} of $D$ if and only if
it admits a neighborhood in which $\partial D$ is $\mathcal C^{\alpha}$-smooth.

Recall that a $\mathcal C^1$-smooth boundary point $\zeta$ of a domain $D$ in $\R^n$ is said to be
{\it Dini-smooth} if the inner unit normal vector $n$ to $\partial D$ near $\zeta$
is a Dini-continuous function. This means that there exists a neighborhood $U$ of $\zeta$ such that
$\ds\int_0^1\frac{\om(t)}{t}dt<+\infty,$ where
$$\om(t)=\om(n,\partial D\cap U,t):=\sup\{||n_x-n_y||:||x-y||<t,\ x,y\in \partial D\cap U\}$$
is the respective modulus of continuity.

If $\ds\int_0^1\om(t)\frac{\log t}{t}dt>-\infty,$ then the point $\zeta$ is called {\it log-Dini smooth.}

The following relations between different notions of smoothness are clear:
${\mathcal C^{1,\varepsilon}}\Rightarrow\mbox{log-Dini}\Rightarrow\mbox{Dini}\Rightarrow{\mathcal C^1}.$

\begin{theorem}\label{NA}\cite[Theorem 7]{NA} Let $\zeta$ be a Dini-smooth boundary point of a domain $D$ in $\R^n.$
Then for any constant $\ds c>1+\frac{\sqrt 2}{2}$ there exists a neighborhood $U$ of $\zeta$ such that
$$h_D(a,b)\le 2\log\left(1+\frac{c||a-b||}{\sqrt{d_D(a)d_D(b)}}\right),\quad a,b\in D\cap U.$$
\end{theorem}

Since $h_D$ is an inner metric, we get an upper bound of $h_D,$ similar to the lower bound from
Proposition \ref{GHM}.

\begin{corollary}\cite[Corollary 8]{NA} Let $D$ be a Dini-smooth bounded domain in $\R^n.$
Then there exists a constant $c>0$ such that
$$h_D(a,b)\le 2\log\left(1+\frac{c||a-b||}{\sqrt{d_D(a)d_D(b)}}\right),\quad a,b\in D.$$
\end{corollary}

Set now $\ds s_D(a,b)=2\sinh^{-1}\frac{||a-b||}{2\sqrt{d_D(a)d_D(b)}}$
$$=2\log\frac{||a-b||+\sqrt{||a-b||^2+4d_D(a)d(b)}}{2\sqrt{d_D(a)d_D(b)}},\quad a,b\in D.$$

Note that $h_D=s_D$ if $D$ is a half-space in $\R^n$ (cf. \cite[(2.8)]{Vuo}).

The following sharp result holds in the $\mathcal C^1$-smooth case.

\begin{proposition}\label{C1}\cite[Proposition 6\hskip1pt(a)]{NA} If $\zeta$ is a $\mathcal C^1$-smooth
boundary point of a domain $D$ in $\R^n,$ then
$$\lim_{\substack{a,b\to\zeta\\a\neq b}}\frac{h_D(a,b)}{s_D(a,b)}=1.$$
\end{proposition}

Since the proof of this proposition is not long, we shall include it for completeness.

\begin{corollary}\cite[Proposition 6\hskip1pt(b) and p.\hskip2pt3]{NA} If $D$ is a $\mathcal C^1$-smooth
bounded domain in $\R^n,$ then
$$q_D(a,b)=\left\{\begin{array}{ll}\ds\frac{h_D(a,b)}{s_D(a,b)},&a,b\in D,\ a\neq b\\
1,&\mbox{otherwise}\end{array}\right.$$
is a continuous function on $\R^n\times\R^n.$
\end{corollary}

The main goal of this paper is to prove the following result related to Proposition \ref{C1}.

\begin{theorem}\label{main} If $\zeta$ is a $\mathcal C^{1,1}$-smooth boundary point of a domain $D$
in $\Bbb R^n$, then
$$\lim_{a,b\to\zeta}(h_D(a,b)-s_D (a,b))=0.$$
\end{theorem}

Note that Theorem \ref{main} and Proposition \ref{C1} say the same only if $s_D$ and
$1/s_D$ are bounded.

The assumption about regularity in Theorem \ref{main} can be weakened in the plane.

\begin{proposition}\label{plane} If $\zeta$ is a log-Dini
smooth boundary point of a domain $D$ in $\Bbb R^2,$ then
$$\lim_{a,b \to \zeta}\left(h_D(a,b)-s_D (a,b)\right)=0.$$
\end{proposition}

The above results imply the following optimal version of Theorem \ref{NA}.

\begin{corollary}\label{cor} Let $\zeta$ be a $\mathcal C^{1,1}$-smooth boundary point of a domain
$D$ in $\Bbb R^n$ or $\zeta$ be a log-Dini smooth boundary point of a domain $D$ in $\Bbb R^2.$
Then for any constant $c>1$ there exists a neighborhood $U$ of $\zeta$ such that
$$h_D(a,b)\le 2\log\left(1+\frac{c||a-b||}{\sqrt{d_D(a)d_D(b)}}\right),\quad a,b\in D\cap U.$$
\end{corollary}

The rest of the paper is organized as follows: Section 2 contains the proofs
of Propositions \ref{C1}, \ref{plane} and Corollary \ref{cor}. Section 3 contains
the proof of Theorem \ref{main}. It should be mentioned that the three proofs use
different flattening maps. Section 4 contains the proof of a result analogous to
Corollary \ref{cor} for the Kobayashi distance.

\section{Proofs of Propositions \ref{C1}, \ref{plane} and Corollary \ref{cor}}

\noindent{\it Proof of Proposition \ref{C1}}. After translation and rotation,
we may assume that $\zeta=0$ and that there is a neighborhood $U$ of $0$ such that
$$
D':=D\cap U=\{x\in U: r(x):=x_1+f(x')>0\},
$$
where points of $\R^n$ are denoted by $x=(x_1,x')$, with $x'\in \R^{n-1}$, and
$f$ is a $\mathcal C^1$-smooth function in $\R^n$ with $f(0)=0$ and $\nabla f(0)=0.$

Let $c>1$ and $\theta(x)=(r(x),x').$ We may shrink $U$ such that
\begin{equation}\label{c}
c^{-1}||x-y||\le||\theta(x)-\theta(y)||\le c||x-y||,\quad x,y\in U.
\end{equation}

Choose now a neighborhood $V\subset U$ of $0$ such that $d_{D'}=d_D$ on $D\cap V.$
The regularity of $D$ implies that it is a {\it uniform domain} near $\zeta$ in the
sense of \cite{GO}. Using, for example, \cite[Corollary 2]{GO}, one can find a neighborhood
$W\subset V$ of $0$ such that any geodesic joining points in $\tilde D=D\cap W$ is contained
in $D\cap V.$ Then $h_D=h_{D'}$ on $\tilde D^2.$

Set $\Bbb R^n_+=\{x\in\R^n:x_1>0\}.$ Using the above arguments,
we may shrink $W$ such that $h_{\RRR}=h_{\theta(D')}$ on $(\theta(\tilde D))^2.$

On the other hand, \eqref{c} implies that (cf. \cite[Exercise 3.17]{Vuo})
$$c^{-2}h_{D'}(z,w)\le h_{\theta(D')}(\theta(z),\theta(w))\le c^2h_{D'}(z,w),\quad z,w\in D'.$$

Let $z,w\in \tilde D.$ Then
$$c^{-2}h_D(z,w)\le h_{\RRR}(\theta(z),\theta(w))\le c^2h_D(z,w).$$
Using \eqref{c} again, we get that
\begin{align*}h_{\RRR}(\theta(z),\theta(w))&=2\sinh^{-1}\frac{||\theta(z)-\theta(w)||}{2\sqrt{r_D(z)r_D(w)}}\\
&\le2\sinh^{-1}\frac{c^2||z-w||}{2\sqrt{d_D(z)d_D(w)}}\le c^2s_D(z,w).
\end{align*}
We obtain in the same way that
$$h_{\RRR}(\theta(z),\theta(w))\ge c^{-2}s_D(z,w).$$
So
$$c^{-4}h_D(z,w)\le s_D(z,w)\le c^4h_D(z,w)$$
which implies the desired result.\qed
\medskip

\noindent{\it Proof of Proposition \ref{plane}}. We may find a neighborhood $U$
of $\zeta$ such that $D\cap U$ is a bounded simply connected log-Dini smooth domain.
Using an argument from the previous proof, we may replace $D$ by $D\cap U.$

The Kellogg--Warschawski theorem (cf. \cite[Theorem 3.5]{Pom}) implies that there exists a conformal
map $\tilde f$ from the unit disc $\Bbb D$ to $D$
which extends to a $\mathcal C^1$-diffeomorphism  between
$\overline{\Bbb D}$ to $\overline D$ such that $\tilde f(\zeta)=1$ and
$$|\tilde f'(z)-\tilde f'(w)|\le\om^\ast(|z-w|),\quad z,w\in\Bbb D,$$ where
$\ds\tom^\ast(s)=\int_0^s\frac{\tom(t)}{t}dt +s\int_s^{+\infty}\frac{\tom(t)}{t^2}dt$ ($s\ge 0$)
and $\tom:\R^+\to\R^+$ is a bounded continuous function with
$\ds\int_0^1\tom(t)\frac{\log t}{t}dt>-\infty.$

Then $\ds f(z)=\tilde f\left(\frac{1-z}{1+z}\right)$ maps conformally $\RR$ onto $D$ and
$$|f'(z)-f'(w)|\le\om^\ast(|z-w|),\quad z,w\in G=\RR\cap\Bbb D,$$ where
$\om^\ast$ is defined in the same way as $\tom^\ast.$

The equality
$$f(w)-f(z)-f'(z)(w-z)=(w-z)\int_{0}^1\left(f'(z+t(w-z))-f'(z)\right)dt$$
implies that $$|f(w)-f(z)-f'(z)(w-z)|\le|w-z|\om^\ast(|w-z|)$$
(since $\om^\ast$ is an increasing function). It follows that
\begin{equation}\label{d}
|d_D(f(z))-|f'(z)|d_{\RR}(z)|\le d_{\RR}(z)\om^\ast(d_{\RR}(z)),\quad z\in G.
\end{equation}

Since $D$ is a uniform domain, there exists a neighborhood $V$ of $\zeta$ such that
any geodesic $\g$ joining points $a=f(\alpha)$ and $b=f(\beta)$ in $D\cap V$ is contained in $f(G).$
It follows by \eqref{d} that one may find a constant $C>0$ (independent of $a$ and $b$)
such that
$$h_{\RR}(\alpha,\beta)\le\int_{f^{-1}\circ\g}\frac{|du|}{d_{\RR}(u)}\le\int_\g\frac{|dv|}{d_D(v)}+
C\int_\g\frac{\om^\ast(d_D(v))}{d_D(v)}|dv|.$$

The first summand is equal to $h_D(a,b).$

We claim that the second summand tends to 0 as $a,b\to\zeta.$ Indeed,
denote by $t$ the natural parameter of $\g$ by arc length and by $l=l(\g)$ the Euclidean length
of $\g.$ Since $D$ is a uniform domain, then \cite[Corollary 2]{GO} provides a constant
$c>0$ (independent of $a$ and $b$) such that $c\cdot l\le|a-b|$ and
$d_D(\g(t))\ge c\cdot \max\{t,l-t\}.$ Using that $\ds\frac{\om^\ast(s)}{s}$ is a decreasing function,
we get
$$\int_\g\frac{\om^\ast(d_D(v))}{d_D(v)}|dv|\le\frac{2}{c}\int_0^{cl/2}\frac{\om^\ast(t)}{t}dt.$$
It is easy to check the log-Dini condition for $\om$ is is equivalent to the fact that the last integral
tends to $0$ as $l\to 0$ which implies our claim.

Hence
$$\liminf_{a,b\to\zeta}(h_D(a,b)-h_{\RR}(\alpha,\beta))\ge 0.$$

The opposite inequality
$$\limsup_{a,b\to\zeta}(h_D(a,b)-h_{\RR}(\alpha,\beta))\le 0$$
follows in the same way by taking the geodesic joining $\alpha$ and $\beta.$

Using \eqref{d}, we have that
\begin{equation}
\label{asymp}
\lim_{\substack{a,b\to\zeta\\a\neq b}}\frac{|a-b|}{2\sqrt{d_D(a)d_D(b)}}\cdot
\frac{2\sqrt{d_{\RR}(\alpha)d_{\RR}(\beta)}}{|\alpha-\beta|}=1.
\end{equation}
Since $h_{\RR}=s_{\RR}$ and 
$\sinh^{-1}qt < \log q + \sinh^{-1}t $ for $q>1,$ $t>0,$ then
$$\lim_{a,b\to\zeta}(s_D(a,b)-h_{\RR}(\alpha,\beta))=0$$
which completes the proof.\qed
\medskip

\noindent{\it Proof of Corollary \ref{cor}}. We may assume that $c=2c'-1\in(1,3].$
By Proposition \ref{C1}, Theorem \ref{main} and Proposition \ref{plane},
one may find a neighborhood $U$ of $\zeta$ such that for $a,b\in D\cap U,$
$$h_D(a,b)\le c's_D(a,b),\quad h_D(a,b)\le s_D(a,b)+\log c'.$$
Then the result follows by the inequalities $\ds \sinh^{-1}\frac{t}{2}<\log(1+t)$ ($t>0$),
$(1+t)^{c'}<1+ct$ ($0<t<1$) and $c'(1+t)<1+ct$ ($t>1$).\qed

\section{Proof of Theorem \ref{main}}

Theorem \ref{main} will follow from Propositions \ref{liminf} and \ref{limsup} below.

For convenience, we assume that $D$ is a domain in $\R^{n+1}$ ($n\ge 1$).

We first localize the problem.  We choose local coordinates so that $\zeta = 0$
and $T_0\partial D = \{0\} \times \R^n$.

Denote points in $\R^{n+1}$ by $\bar x = (x_0, x) \in \R \times \R^{n}$.
We also write $\R^{n+1}_+= \{\bar x\in\R^{n+1}:x_0>0\}$.

There are a ball $\mathcal U \subset \R^{n+1}$ centered at $(0,0)$ and
a function $f \in \mathcal C^{1,1}(\mathcal U\cap\R^n,\R)$
such that $f(0)=0$ and $Df(0)=0$ and
\begin{equation}
\label{localform}
D\cap\mathcal U= \left\{\bar x\in\mathcal U:x_0 > f(x) \right\}.
\end{equation}

By shrinking the radius of $\mathcal U$ further we may assume that the projection which to
$\bar x\in \mathcal U \cap D$ associates $\pi(\bar x)$, the closest point
in $\partial D$ is well-defined, and that $\mathcal U \subset \pi^{-1} (\mathcal U \cap D)$
(see \cite[Lemma 4.11]{Fed}, or the proof of Lemma \ref{norm} (1) below).

\begin{proposition}\label{liminf}
$\ds\liminf_{a,b \to 0} \left( h_D(a,b) - s_D (a,b) \right) \ge 0.$
\end{proposition}

We can define a map $\varphi$ on $\mathcal U$ by
$$\varphi  (\bar x) = \left( f(x), x \right) + x_0 n_x,$$
where $n_x$ is the inward unit normal to $\partial D$ at the point $( f(x), x)$.

\begin{lemma}\label{norm} (1) There exists a ball $\mathcal U_0 \subset \mathcal U$
centered at $0$ such that $\varphi|_{\mathcal U_0}$ is a bilipschitz homeomorphism
and for any $\bar x \in \mathcal U_0 \cap \R^{n+1}_+$,
$$d_D \varphi (\bar x)= \| \varphi (\bar x) - (f(x), x)\| = x_0.$$

\noindent(2) Furthermore, if $f \in \mathcal C^{\alpha}(\mathcal U\cap\R^n, \R)$, for some
$\alpha\ge 2$, then $\varphi|_{\mathcal U_0}$ is a $\mathcal C^{\alpha-1}$-diffeomorphism,
and there exists a ball $\mathcal U_1 \subset \mathcal U_0$ centered at $0$ and a
constant $C>0$ such that for any $\bar x \in \mathcal U_1 \cap \R^{n+1}_+$ and any
vector $v\in \R^{n+1}$,
$$\| D\varphi(\bar x)\cdot v\|\ge(1-Cx_0)\|v\|.$$
where $D\varphi (\bar x)$ stands for the differential of $\varphi$ taken
at the point $\bar x$.

\noindent(3) In the general case where $f \in \mathcal C^{1,1}(\mathcal U\cap\R^n, \R)$,
then there exists a $C>0$ such that
for any $\mathcal C^1$ curve $\gamma: [t_1,t_2] \longrightarrow \mathcal U_1 \cap  \R^{n+1}_+$,
$\varphi \circ \gamma$ is rectifiable and for any $F \in \mathcal C([t_1,t_2], \R_+)$,
$$
\int_{t_1}^{t_2} F(t) |d \varphi \circ \gamma (t)|
\ge
\int_{t_1}^{t_2} F(t) |d  \gamma (t)| - C \int_{t_1}^{t_2} F(t) d_D (\gamma (t)) |d  \gamma (t)|.
$$
\end{lemma}

\begin{proof}
Part (1) of the lemma is classical (see \cite[Theorem 4.8]{Fed}).
The main point is to prove that the domain has positive {\it reach,} that is to say that
there exists $\delta>0$ such that if $x \in D$ and $d_D(x)<\delta$, then this distance is attained
at a single point, which will be the intersection of $\partial D$ and the unique normal line to it
containing $x$ (see \cite{Fed}).
In other words, for $x \in \mathcal U$ well chosen and $x_0<\delta$, $\varphi$ is one-to-one.

We quickly recall the proof.  Suppose $\| \nabla f (x) - \nabla f(x')\| \le L \|x-x'\|$
for $(0,x), (0,x') \in \mathcal U_1$,  then, taking without loss of generality the projection to
$\partial D$ to be $(0,0)$, for some $\theta \in (0,1)$,
\begin{multline*}
\left\| (y_0,0) - (f(x),x)\right\|^2 = y_0^2  - 2 y_0 \nabla f (\theta x) \cdot x + f(x)^2 + \|x\|^2
\\
\ge y_0^2+ \|x\|^2 - 2 y_0 L \|x\|^2 >y_0^2
\end{multline*}
for $y_0 < 1/2L$ and $x\neq 0$.

Notice that a lemma in \cite[Appendix]{GT}, explained in detail in \cite{KP},
shows that even though $n_x$ can only be expected to be continuous with bounded derivatives,
and in general of class $\mathcal C^{\alpha-1}$ when $\varphi \in \mathcal C^{\alpha}$,
the function $\bar x\mapsto d_D(\bar x)$ has the same regularity as $\varphi$.

We now prove part (2). Let $(e_0, e_1, \dots , e_n)$ be the standard basis of $\R^{n+1}$.
Let $\ds \tilde e_j = \frac{\partial f}{\partial x_j} (x) e_0 + e_j$, for $1\le j\le n$. They
form a basis of the tangent space to $\partial D$ at $(x,f(x))$
and $\langle n_x, \tilde e_j \rangle=0$ for $1\le j\le n$.

Then $D\varphi (\bar x) \cdot e_0 = n_x$, and
$\ds D\varphi (\bar x) \cdot e_j = \tilde e_j + x_0 \frac{\partial n_x}{\partial x_j}$, for $1\le j\le n$.

Given $v= \sum_0^n v_j e_j$,
$$
D\varphi (\bar x) \cdot v = \left( v_0 n_x + \sum_1^n v_j \tilde e_j \right) + x_0 \sum_1^n v_j
\frac{\partial n_x}{\partial x_j} =: V_1 + V_0.
$$
Clearly, $\| V_0 \| = O(x_0) \|v\|$. By the orthogonality of $n_x$ to the tangent space,
\begin{multline*}
\| V_1 \|^2 = v_0^2 + \left\|\sum_1^n v_j \tilde e_j  \right\|^2
= v_0^2 + \left\|\sum_1^n v_j e_j  + \left(\sum_1^n \frac{\partial f}{\partial x_j} (x)\right)e_0\right\|^2
\\
= v_0^2 + \sum_1^n v_j^2 + \left|\sum_1^n \frac{\partial f}{\partial x_j} (x)\right|^2 \ge \|v\|^2.
\end{multline*}

In the case where $f \in \mathcal C^{1,1}$, then $\varphi \circ \gamma$ is only a Lipschitz map.
By Rademacher's theorem (see e.g. \cite[Theorem 3.1.6]{Fe2}), it is almost everywhere differentiable
and the fundamental theorem of calculus holds. We then perform the same calculation as in case (2),
where the integrands are defined a.e.
\end{proof}

\noindent{\it Proof of Proposition \ref{liminf}}. Using Lemma \ref{norm},
the proof repeats the second part of the proof of Proposition \ref{plane}.
Suppose that $\zeta=0$ and that the domain $D$ is given by a local representation as above.
We may assume that the points $a,b\in D$ are in a small
enough neighborhood of $0$ so that the geodesic $\g$ which joins them is entirely
contained in the range of invertibility of $\varphi$ and Lemma \ref{norm} holds; we write
$a=\varphi(\bar \alpha)$, $b=\varphi(\bar \beta)$, $\g=\varphi(\tilde\g)$,
where $\tilde\g$ is an arc in $\R^{n+1}_+$. Then
$$
h_D(a,b) =
\int_\g\frac{||du||}{d_D(u)}\ge\int_{\tilde\g}\frac{||dv||}{d_{\R^{n+1}_+}(v)}
-C \cdot l(\tilde\g)\ge h_{\R^{n+1}_+}(\bar \alpha,\bar \beta)-C'||\bar\alpha-\bar\beta||,
$$
where $C'>0$ is a constant independent of $a$ and $b.$ Note that
$h_{\R^{n+1}_+}=s_{\R^{n+1}_+}.$ Since the differential of
$\varphi$ at $\bar x$ tends to the identity as $x\to 0,$ it follows that
$$\lim_{a,b\to\zeta}(s_{\R^{n+1}_+}(\bar \alpha,\bar \beta)-s_D(a,b))=0$$
which completes the proof.\qed

\begin{proposition}\label{limsup}
$\ds\limsup_{a,b \to 0} \left( h_D(a,b) - s_D (a,b) \right) \le 0.$
\end{proposition}

The proof is similar to that of Proposition \ref{liminf}, using a modification
of the map $\varphi$ which depends on $a$ and $b.$

\begin{proof}
We again assume that $a, b \in D$, and the geodesic
connecting them, all lie in a neighborhood of $\zeta$ small enough
so that any point in it has a unique closest point on $\partial D$.
Let  $a',b'$ be the respective
closest points. We take new coordinates (and obtain a new function $f$)
so that $a'=0$ (instead of $\zeta=0$ as in the proof of Proposition \ref{liminf}) and
$$D\cap\mathcal U=\{\bar x\in\mathcal U:x_0>f(x_1,\dots,x_n)\}.$$
We may also assume that
$b'_2=\dots=b'_n=0.$ Shrinking the radius $r$ of $\mathcal U,$
we may replace $x_1$ by $\sigma_1(x_1)$
such that for $\sigma=(f(\sigma_1,0,\dots,0),\sigma_1, 0,\dots,0)$
one has $||\sigma'||=1$ (in other words, $\sigma$ is parametrized by arc length).
Note that $r$ can be chosen independently of $a$ and $b.$
Let $\ell$ be the length of the curve $\sigma$ from $a'$ to $b'$, so that
$\sigma (0)=a'$, $\sigma (\ell)=b'$.

Consider the map $\varphi$ from $\RR$ (near $0$) to $D$ defined by
$$\varphi(x_0,x_1) = \sigma (x_1) + x_0 n_{\sigma (x_1)},$$
where $n_{\sigma (x_1)}$ is the inward unit normal to $\partial D$ at the point
$\sigma (x_1).$ Then $d_D(\varphi (\bar x))=x_0$ if $x_0$ is small enough,
and if $\alpha = (d_D(a), 0)$ and $\beta = (d_D(b), \ell)$,
we have $\varphi (\alpha)=a$, $\varphi (\beta)=b$.

\begin{lemma}
\label{norm2} There exist a neighborhood $U$ of $\zeta,$ a neighborhood
$V$ of $0$ and a constant $C>0$ such that for any
$a,b \in D\cap U$ and $\bar x \in\RR\cap V$
and any vector $v\in \R^2$, then $\alpha, \beta \in V$ and
$$\| D\varphi(\bar x)\cdot v\|\le(1+Cx_0)\|v\|.$$
\end{lemma}

\begin{proof}
As in the proof of Lemma \ref{norm} (2), in the $\mathcal C^2$-smooth case,
$$
D \varphi (\bar x) \cdot e_0 = n_{\sigma (x_1)}, \quad
D \varphi (\bar x) \cdot e_1 = \sigma' (x_1) + x_0 \frac{\partial n_{\sigma (x_1)}}{\partial x_1}.
$$
Because $||\sigma'||=1$ and is tangent to $\partial D$,
$(\sigma'(x),n_x)$ form an orthonormal system,
so that $D\psi (\bar x)$ differs from a linear isometric embedding by a term bounded by
$\ds \left\| \frac{\partial n_{\sigma (x_1)}}{\partial x_1} \right\| x_0 $.

Geometric considerations show that
$\ds \left\| \frac{\partial n_{\sigma (x_1)}}{\partial x_1} \right\|\le\frac{1}{R} $
whenever there exist two balls $B_1, B_2$ of radius $R$, tangent to each side of
$\partial D$ at $\sigma (x_1)$.
The argument in the proof of Lemma \ref{norm} (1) shows there
exists $\delta >0$ (depending only on the neighborhood $\mathcal U_0$ mentioned in that lemma)
such that  there exist two such balls of radius $\delta$ at each point in $\mathcal U_0 \cap \partial D $.

As in the proof of Lemma \ref{norm} (3), the $\mathcal C^{1,1}$-smooth case follows by applying Rademacher's theorem.
\end{proof}

The proof of Proposition \ref{limsup} can be finished similarly to that of Proposition
\ref{liminf}. Let $\g$ be the geodesic joining $\alpha$ to $\beta$ in $\R^2_+.$
Let $U,V$ be as in Lemma \ref{norm2}. Shrinking $V$ if needed so that $\varphi(V)\subset U$,
we have $d_{D}(\varphi(u))=d_{\RR}(u)$ for any $u\in\g$.
Since $\varphi \circ \g $ is a curve joining $a$ to $b$
in $D$, using Lemma \ref{norm2}, we get
\begin{multline*}
h_D(a,b) \le
 \int_0^\ell \frac{\|D\varphi (\gamma(t))\cdot \gamma'(t) \|}{d_{D}( \varphi \circ \gamma(t))} dt\\
\le h_{\RR} (\alpha,\beta)+Cl(\g)<s_{\RR} (\alpha,\beta)+C\pi||\alpha-\beta||
\end{multline*}
(here $\pi$ is the Ludolphine number, not the projection). The differential of $\varphi$
is close to a linear isometric embedding of $\R^2$ in $\R^{n+1}$ and hence
we have the asymptotic relation \eqref{asymp} and
$$\lim_{a,b\to\zeta}(s_{\RR}(\alpha,\beta)-s_D(a,b))=0,
$$
which completes the proof.
\end{proof}

\section{An upper estimate for the Kobayashi distance}

Let $D$ be a domain in $\Bbb C^n.$ The Kobayashi (pseudo) distance
$k_D$ is
obtained from the Lempert function
\begin{multline*}
l_D(a,b)=\inf\{\tanh^{-1}|\alpha|:\exists\varphi\in\mathcal O(\Bbb D,D)
\hbox{ with }\varphi(0)=a,\varphi(\alpha)=b\},\\a,b\in D.
\end{multline*}
The Lempert function does not always satisfy the triangle inequality, but
setting
$$
k_D(a,b):= \inf \left\{ \sum_{j=0}^{m-1} l_D(a_j,a_{j+1}): a_j \in D, a_0=a,a_m=b, m\ge 1 \right\},
$$
one does obtain a (pseudo) distance, which is the largest that is
dominated by $l_D$.

Recall that $k_D$ is the integrated form of the Kobayashi (pseudo) metric
\begin{multline*}
\kappa_D(a;X)=\inf\{|\alpha|:\exists\varphi\in\mathcal O(\Bbb D,D)\hbox{
with }\varphi(0)=a,\alpha\varphi'(0)=X\},\\a\in D,\ X\in\Bbb C^n.
\end{multline*}

Note that Theorem \ref{NA} and Proposition \ref{plane} (even in the
Dini-smooth case) hold for $2k_D$ instead of $h_D$ (see \cite[Theorem 7]{NA} and
\cite[Proposition 6]{NTA}). Moreover, the following result corresponds to
Proposition \ref{C1}.

\begin{proposition}\cite[Proposition 5\hskip1pt(a)]{NA} If $\zeta$ is a $\mathcal C^1$-smooth
boundary point of a domain $D$ in $\Bbb C^n,$ then
$$\limsup_{\substack{a,b\to\zeta\\a\neq b}}\frac{2k_D(a,b)}{h_D(a,b)}\le 1.$$
\end{proposition}

It turns out that Corollary \ref{cor} also holds for $2k_D$ instead of $h_D.$ This gives
the optimal version of \cite[Proposition 2.5]{FR} in the $\mathcal C^{1,1}$-smooth case.

\begin{proposition} Let $\zeta$ be a $\mathcal C^{1,1}$-smooth boundary point of a domain $D$
in $\Bbb C^n$ or $\zeta$ be a log-Dini smooth boundary point of a domain $D$ in $\Bbb C.$
Then for any constant $c>1$ there exists a neighborhood $U$ of $\zeta$ such that
$$k_D(a,b)\le\log\left(1+\frac{c||a-b||}{\sqrt{d_D(a)d_D(b)}}\right),\quad a,b\in D\cap U.$$
\end{proposition}

\noindent{\it Proof.} Having in mind Corollary \ref{cor}, it is enough to show that
$$\limsup_{\substack{a,b\to\zeta\\a\neq b}}\frac{2k_D(a,b)-h_D(a,b)}{||a-b||}<+\infty.$$

Since $k_D$ is the integrated form of $\kappa_D$ and the lengths of the
quasi-hyperbolic geodesics joining points in $D$ near $\zeta$ are bounded up to a multiplicative
constant by the Euclidean distances between the points, the last inequality will be a consequence
of the following one:
$$\limsup_{\substack{a\to\zeta\\||X||=1}}\left(2\kappa_D(a;X)-\frac{1}{d_D(a)}\right)<+\infty.$$

To see this, note that there exists an $r>0$ such that any $a\in D$ near $\zeta$ is contained in
a (unique) ball $\Bbb B_n(\tilde a,r)\subset D$ with $r-||a-\tilde a||=d_D(a)$ (the inner ball condition).
It remains to use that for such an $a$ and $||X||=1$ one has that
$$\kappa_D(a;X)\le \kappa_{\Bbb B_n(\tilde a,r)}(a;X)\le\frac{r}{r^2-||a-\tilde a||^2}
<\frac{1}{2d_D(a)}+\frac{1}{4r}.\quad\qed$$

\end{document}